\documentclass[reqno]{amsart}

\usepackage{amssymb}
\usepackage{graphicx}
\usepackage{amscd}
\usepackage[pagebackref]{hyperref}
\usepackage{color}
\usepackage{tabularx}
\usepackage[table]{xcolor}
\usepackage{float}
\usepackage{graphics,amsmath,amssymb}
\usepackage{amsthm}
\usepackage{amsfonts}
\usepackage{latexsym}
\usepackage{epsf}
\usepackage{xifthen}
\usepackage{mathrsfs}
\usepackage{dsfont}
\usepackage{makecell}
\usepackage{subfig}
\usepackage{amsmath}
\allowdisplaybreaks[4]
\usepackage{listings}
\usepackage{etoolbox}
\usepackage{fancyhdr}
\usepackage{pdflscape}
\usepackage[title,toc,titletoc]{appendix}
\usepackage{enumitem}
\usepackage[noadjust]{cite}
\usepackage{tikz}
\usetikzlibrary{automata,positioning,arrows}
\usepackage{young}
\usepackage[object=vectorian]{pgfornament} %%  http://altermundus.com/pages/tkz/ornament/index.html
\usepackage{lipsum,tikz}
\usepackage{multirow}
\usepackage[OT2,T1]{fontenc}
\usepackage{mathtools}

\hypersetup{
	colorlinks=true, %set true if you want colored links
	linktoc=all, %set to all if you want both sections and subsections linked
	linkcolor=blue} %choose some color if you want links to stand out

\numberwithin{equation}{section}

\theoremstyle{theorem}
\newtheorem{theorem}{Theorem}[section]
\newtheorem*{theorem*}{Theorem}

\newtheorem{corollary}[theorem]{Corollary}
\newtheorem{lemma}[theorem]{Lemma}

\newtheorem{conjecture}[theorem]{Conjecture}

\newtheorem{innercustomgeneric}{\customgenericname}
\providecommand{\customgenericname}{}
\newcommand{\newcustomtheorem}[2]{%
	\newenvironment{#1}[1]
	{%
		\renewcommand\customgenericname{#2}%
		\renewcommand\theinnercustomgeneric{##1}%
		\innercustomgeneric
	}
	{\endinnercustomgeneric}
}
\newcustomtheorem{ctheorem}{Theorem}
\newcustomtheorem{clemma}{Lemma}

\theoremstyle{definition}

\newtheorem*{example*}{Example}
\newtheorem*{examples*}{Examples}

\newtheorem*{remark*}{Remark}
\newtheorem*{remarks*}{Remarks}
\newtheorem*{note*}{Note}

\newtheoremstyle{named}{}{}{\itshape}{}{\bfseries}{.}{.5em}{#1\thmnote{ #3}}
\theoremstyle{named}

\makeatletter
\patchcmd{\subsection}{\bfseries}{\bfseries\boldmath}{}{}
\makeatother

\newcommand{\arxiv}[1]{\href{https://arxiv.org/abs/#1}{arXiv:#1}}

\newcommand{\qbinom}[2]{{#1\brack #2}}

\newcommand{\LHS}{\operatorname{LHS}}
\newcommand{\RHS}{\operatorname{RHS}}

\newcommand{\ii}{\operatorname{i}}
\newcommand{\ee}{\operatorname{e}}
\newcommand{\dd}{\operatorname{d}}

\newcommand{\cI}{\mathcal{I}}
\newcommand{\Mac}{\mathsf{Mac}}

\title[Fermionic formulas for the Macdonald index]{On conjectural fermionic formulas for the Macdonald index in Argyres--Douglas theories}

\author[S. Chern]{Shane Chern}
\address[S. Chern]{Fakult\"at f\"ur Mathematik, Universit\"at Wien, Oskar-Morgenstern-Platz 1, Wien 1090, Austria}
\email{chenxiaohang92@gmail.com, xiaohangc92@univie.ac.at}

\author[C. Tran]{Chanh Tran}
\address[C. Tran]{Prompt Inversion AI}
\email{chanh@promptinversion.ai}

\author[T. Wakhare]{Tanay Wakhare}
\address[T. Wakhare]{{}\textsuperscript{(1)}Department of Electrical Engineering and Computer Science, MIT, Cambridge, MA 02139, USA \newline \indent{}\textsuperscript{(2)}Prompt Inversion AI}
\email{twakhare@mit.edu}

\date{}



\begin{document}
	
\sloppy

\begin{abstract}
	We prove a fermionic--bosonic duality relation for the Macdonald index in Argyres--Douglas theories of type $(A_1, D_{2k+1})$, thereby yielding a conjectural fermionic formula due to Andrews et al. Our duality is built upon a new conjugate Bailey pair to be established using techniques from orthogonal polynomials and basic hypergeometric series. In addition, this fermionic formula implies another sum-like expression independently conjectured by Andrews et al.~and Kim et al.~for the same Macdonald index.
\end{abstract}

\maketitle

\section{Introduction}

Originating in \cite{AD1995}, the \emph{Argyres--Douglas theories} are families of four-dimensional $\mathcal{N} = 2$ superconformal field theories to describe the interactions of massless, mutually non-local BPS particles. Geometrically, they are realized by compactifying six-dimensional $\mathcal{N} = (2,0)$ theories on a Riemann surface with irregular punctures \cite{WX2016, Xie2013}. In particular, these theories have marginal deformations and the Coulomb branch spectrum is of fractional scaling dimension; see \cite{BBMN2021, GHMM2025, GMS2021} for more background in physics and recent advances on this topic.

Due to the strong-coupling effects, no Lagrangian description exists in any duality frame for Argyres--Douglas theories, and there is a lack of \emph{direct} ways for computing the superconformal indices, such as Schur index, Hall--Littlewood index, and Macdonald index. For the latter challenge, Buican and Nishinaka~\cite{BN2016c} made the very first progress on the Macdonald index in $(A_1, A_{2k+1})$ and $(A_1, D_{2k})$ Argyres--Douglas theories. In addition, they~\cite{BN2016} proposed an Ansatz for the topological quantum description of Schur index, agreeing with the $S$-duality of Argyres--Douglas theories \cite{BGNP2015, XY2018}, and this construction was subsequently shown to be consistent with $S^1$ reductions of Argyres--Douglas theories \cite{BN2016b}. Alternatively, C\'ordova and Shao~\cite{CS2016} predicted the Schur index according to the correspondence between superconformal field theories and vertex operator algebras, and their results match those in \cite{BLL+2015}. Along these directions, Song~\cite{Son2016} made further headway on the computation of the three types of superconformal indices.

In this work, we focus on the Macdonald index in Argyres--Douglas theories of type $(A_1, D_{2k+1})$. Briefly speaking, the \emph{Macdonald index} $\cI_{\Mac}^{(2k+1)}(z,t;q)$, carrying fugacities $z$, $t$, and $q$, comes from a trace formula \cite{KMMR2007} for certain BPS states associated with the $(A_1, D_{2k+1})$ theory. This Macdonald index reduces to the Schur index at the $q = t$ case, and to the Hall--Littlewood index when $q = 0$.

There is a \emph{bosonic}\footnote{Bosonic and fermionic representations are two ways to describe the same physical system: the former treats the system as a collection of independent particles, usually leading to single sum-like expressions mathematically, while the latter looks at the exclusion principle of the particles, resulting in multisums.} sum-like representation for $\cI_{\Mac}^{(2k+1)}(z,t;q)$ \cite[p.~5, eq.~(2.7)]{KKS2026}:
\begin{align}\label{eq:Mac-bosonic}
	\cI_{\Mac}^{(2k+1)}(z,t;q) &= \frac{1}{(t,tz^2,tz^{-2};q)_\infty}\sum_{n\ge 0} (-1)^n t^{(k+1)n} q^{kn^2+\binom{n}{2}} \frac{(q^{n+1};q)_n (t^2q^{2n};q)_\infty}{(tq^n;q)_n (tq^{2n+1};q)_\infty}\notag\\
	&\quad\times \sum_{j=0}^{2n} \frac{(t;q)_j (t;q)_{2n-j}}{(q;q)_j (q;q)_{2n-j}}z^{2j-2n}.
\end{align}
Throughout, we adopt the conventional \emph{$q$-Pochhammer symbols}:
\begin{align*}
	(a;q)_\infty := \prod_{k\ge 0} (1-aq^k),\qquad\qquad (a;q)_n := \frac{(a;q)_\infty}{(aq^n;q)_\infty},
\end{align*}
with the compact notation
\begin{align*}
	(a_1,\ldots,a_M;q)_\infty &:= (a_1;q)_\infty \cdots (a_M;q)_\infty,\\
	(a_1,\ldots,a_M;q)_n &:= (a_1;q)_n \cdots (a_M;q)_n.
\end{align*}
In addition, we need the \emph{$q$-binomial coefficients}:
\begin{align*}
	\qbinom{M}{N}_q:=\begin{cases}
		\dfrac{(q;q)_M}{(q;q)_N(q;q)_{M-N}}, & \text{if $0\le N\le M$},\\[10pt]
		0, & \text{otherwise}.
	\end{cases}
\end{align*}

Recently, Andrews et al.~\cite{ABBST2025}, on the other hand, made the following conjecture\footnote{The original formulation in \cite[Conjecture~4]{ABBST2025} goes to \eqref{eq:ABBST-conj} after a simple substitution of indices.} about the \emph{fermionic} expression for $\cI_{\Mac}^{(2k+1)}(1,t;q)$.

\begin{conjecture}[Andrews et al.~{\cite[Conjecture~4]{ABBST2025}}]
	\begin{align}\label{eq:ABBST-conj}
		\cI_{\Mac}^{(2k+1)}(1,t;q) \overset{?}{=} \sum_{n_k\ge \cdots\ge n_1 \ge 0} \frac{t^{n_1+\cdots+n_{k}} q^{n_1^2+\cdots + n_{k-1}^2}}{(q;q)_{n_k-n_{k-1}}\cdots (q;q)_{n_2-n_1} (q;q)_{n_1}}\sum_{j=0}^{2n_k} \qbinom{2n_k}{j}_q.
	\end{align}
\end{conjecture}

Amongst basic hypergeometric series, especially (finite) Rogers--Ramanujan type identities \cite{Sil2003}, the \emph{fermionic--bosonic duality} plays a significant role. In principle, it refers to equalities of the form
\begin{align*}
	\text{``multiple $q$-summation} = \text{single $q$-summation.''}
\end{align*}
A typical example is
\begin{align*}
	\sum_{n_{k}\ge \cdots\ge n_1 \ge 0} \frac{q^{n_1^2+\cdots + n_{k}^2}}{(q;q)_{n_{k}-n_{k-1}}\cdots (q;q)_{n_2-n_1} (q;q)_{n_1}} = \frac{1}{(q;q)_\infty}\sum_{n= -\infty}^\infty (-1)^n q^{(k+1)n^2 + \binom{n}{2}},
\end{align*}
which produces Andrews' multiple generalization of the first Rogers--Ramanujan identity~\cite[p.~4082, Theorem~1]{And1974} after applying Jacobi's triple product~\cite[p.~357, eq.~(II.28)]{GR2004} to the bosonic sum on the right-hand side.

Now the first objective of this paper is to establish the following fermionic dual for \eqref{eq:Mac-bosonic}:
\begin{align}\label{eq:ABBST-z}
	\cI_{\Mac}^{(2k+1)}(z,t;q) = \sum_{n_k\ge \cdots\ge n_1 \ge 0} \frac{t^{n_1+\cdots+n_{k}} q^{n_1^2+\cdots + n_{k-1}^2}}{(q;q)_{n_k-n_{k-1}}\cdots (q;q)_{n_2-n_1} (q;q)_{n_1}}\sum_{j=0}^{2n_k} \qbinom{2n_k}{j}_q z^{2j - 2n_k},
\end{align}
which reduces to the conjectural formula of Andrews et al.~\cite[Conjecture~4]{ABBST2025}, that is, \eqref{eq:ABBST-conj}, at $z=1$. More precisely, we show the fermionic--bosonic duality:

\begin{theorem}\label{th:fermionic-bosonic-main}
	For every $k\ge 1$,
	\begin{align}\label{eq:fermionic-bosonic-main}
		&\sum_{n_k\ge \cdots\ge n_1 \ge 0} \frac{t^{n_1+\cdots+n_{k}} q^{n_1^2+\cdots + n_{k-1}^2}}{(q;q)_{n_k-n_{k-1}}\cdots (q;q)_{n_2-n_1} (q;q)_{n_1}}\sum_{j=0}^{2n_k} \qbinom{2n_k}{j}_q z^{2j - 2n_k}\notag\\
		&\qquad\qquad = \frac{1}{(t,tz^2,tz^{-2};q)_\infty}\sum_{n\ge 0} (-1)^n t^{(k+1)n} q^{kn^2+\binom{n}{2}} \frac{(q^{n+1};q)_n (t^2q^{2n};q)_\infty}{(tq^n;q)_n (tq^{2n+1};q)_\infty}\notag\\
		&\qquad\qquad\quad\times \sum_{j=0}^{2n} \frac{(t;q)_j (t;q)_{2n-j}}{(q;q)_j (q;q)_{2n-j}}z^{2j-2n}.
	\end{align}
\end{theorem}

Andrews et al.~\cite[eq.~(26)]{ABST2025} and Kim et al.~\cite[p.~8, eq.~(2.26)]{KKS2026} independently conjectured another fermionic sum for $\cI_{\Mac}^{(2k+1)}(z,t;q)$. Our second purpose is to show that this formula, as stated below using an equivalent form:\footnote{The equivalence can be seen by renaming the summation indices in \cite[eq.~(26)]{ABST2025} or \cite[p.~8, eq.~(2.26)]{KKS2026} and invoking the Cartan matrix of the $D_{2k+1}$ Lie algebra. For details, see Appendix~\ref{appx:fer-2-original}.}
\begin{align}\label{eq:KKS-2.26}
	\cI_{\Mac}^{(2k+1)}(z,t;q) &= (t,q;q)_\infty^k \sum_{\substack{r_1,\ldots,r_k\ge 0\\s_1,\ldots,s_k\ge 0}} \frac{t^{\sum_{i=1}^k s_i} q^{\sum_{i=1}^k r_i(s_{i-1}+s_i+1)}}{(t,q;q)_{r_1}\cdots (t,q;q)_{r_k} (q;q)_{s_1}^2 \cdots (q;q)_{s_k}^2}\notag\\
	&\quad\times \sum_{u_1,u_2\ge 0} \qbinom{s_k}{u_1}_q \qbinom{s_k}{u_2}_q z^{2u_1-2u_2},
\end{align}
with $s_0:=0$, can be transformed to that in \eqref{eq:ABBST-z}, thereby linking all three expressions for the Macdonald index in Argyres--Douglas theories of type $(A_1, D_{2k+1})$. In other words, we prove the following identity:

\begin{theorem}\label{th:KKS-2.26-id}
	For every $k\ge 1$,
	\begin{align}\label{eq:KKS-2.26-id}
		&\sum_{n_k\ge \cdots\ge n_1 \ge 0} \frac{t^{n_1+\cdots+n_{k}} q^{n_1^2+\cdots + n_{k-1}^2}}{(q;q)_{n_k-n_{k-1}}\cdots (q;q)_{n_2-n_1} (q;q)_{n_1}}\sum_{j=0}^{2n_k} \qbinom{2n_k}{j}_q z^{2j - 2n_k}\notag\\
		&\qquad\qquad = (t,q;q)_\infty^k \sum_{\substack{r_1,\ldots,r_k\ge 0\\s_1,\ldots,s_k\ge 0}} \frac{t^{\sum_{i=1}^k s_i} q^{\sum_{i=1}^k r_i(s_{i-1}+s_i+1)}}{(t,q;q)_{r_1}\cdots (t,q;q)_{r_k} (q;q)_{s_1}^2 \cdots (q;q)_{s_k}^2}\notag\\
		&\qquad\qquad\quad\times \sum_{u_1,u_2\ge 0} \qbinom{s_k}{u_1}_q \qbinom{s_k}{u_2}_q z^{2u_1-2u_2},
	\end{align}
	where we put $s_0:=0$.
\end{theorem}

This paper is structured as follows. First, in Section~\ref{sec:Bailey-sum}, we review Bailey's method on summations. Then a new conjugate Bailey pair is constructed in Section~\ref{sec:conjugate-Bailey}. In this course, techniques from orthogonal polynomials and basic hypergeometric series are needed, including a basic hypergeometric evaluation to be established in Appendix~\ref{appx:evaluation}. Next, we choose a proper seed Bailey pair in Section~\ref{sec:fermionic-bosonic} to conclude a strengthening of Theorem~\ref{th:fermionic-bosonic-main}. In Section~\ref{sec:KKS-2.26}, we prove Theorem~\ref{th:KKS-2.26-id}. Finally, we provide a well-poised generalization of our new conjugate Bailey pair as closing remarks in Section~\ref{sec:closing}.

\section{Summations \`a la Bailey}\label{sec:Bailey-sum}

In this section, we briefly recall Bailey's theory on summations \cite{Bai1948}, which is streamlined in \cite[Chapter~3]{And1986}. Let $(u_n)_{n\ge 0}$ and $(v_n)_{n\ge 0}$ be generic sequences. Subject to suitable convergence conditions, we assume that sequences $(\alpha_n)_{n\ge 0}$, $(\beta_n)_{n\ge 0}$, $(\gamma_n)_{n\ge 0}$, and $(\delta_n)_{n\ge 0}$ satisfy the following relations for all $n\ge 0$:
\begin{align*}
	\beta_n = \sum_{l=0}^n u_{n-l} v_{n+l} \alpha_l,
\end{align*}
and
\begin{align*}
	\gamma_n = \sum_{l\ge n} u_{l-n} v_{l+n} \delta_l.
\end{align*}
Then \emph{Bailey's transform} \cite[p.~24, Theorem~3.1]{And1986} asserts that
\begin{align}\label{eq:Bailey-trans}
	\sum_{n\ge 0} \alpha_n \gamma_n = \sum_{n\ge 0} \beta_n \delta_n.
\end{align}
A typical choice of the sequences $u_n$ and $v_n$ is
\begin{align*}
	u_n = \frac{1}{(q;q)_n} \qquad \text{and} \qquad v_n = \frac{1}{(tq;q)_n},
\end{align*}
where $t$ is an indeterminate. We say sequences $\alpha_n$ and $\beta_n$ form a \emph{Bailey pair} relative to $(t,q)$ \cite[p.~25, eq.~(3.27)]{And1986} if for each $n\ge 0$,
\begin{align}\label{eq:Bailey-beta}
	\beta_n = \sum_{l=0}^n \frac{\alpha_l}{(q;q)_{n-l} (tq; q)_{n+l}}.
\end{align}
We may boost this Bailey pair to a \emph{Bailey chain} \cite[p.~29, eq.~(3.38)]{And1986}. To be precise, for each $k\ge 1$, the sequences
\begin{align}\label{eq:Bailey-chain-k-a}
	\alpha_n^{(k)} &:= \frac{(b_1^{-1},c_1^{-1},\ldots,b_k^{-1},c_k^{-1};q)_n (b_1c_1\cdots b_kc_k)^n}{(b_1tq,c_1tq,\ldots,b_ktq,c_ktq;q)_n}\cdot (tq)^{kn} \alpha_n
\end{align}
and
\begin{align}\label{eq:Bailey-chain-k-b}
	\beta_n^{(k)} &:= \sum_{0 \le n_0\le \cdots\le n_{k-1}\le n} \frac{(tq)^{n_0+\cdots+n_{k-1}} \beta_{n_0}}{(q;q)_{n-n_{k-1}}(q;q)_{n_{k-1}-n_{k-2}}\cdots (q;q)_{n_1-n_0}}\notag\\
	&\quad\ \times \frac{(b_kc_ktq;q)_{n-n_{k-1}}(b_{k-1}c_{k-1}tq;q)_{n_{k-1}-n_{k-2}}\cdots (b_1c_1tq;q)_{n_1-n_0}}{(b_1tq,c_1tq;q)_{n_1}\cdots (b_{k-1}tq,c_{k-1}tq;q)_{n_{k-1}}(b_{k}tq,c_{k}tq;q)_{n}}\notag\\
	&\quad\ \times (b_1^{-1},c_1^{-1};q)_{n_0}\cdots (b_k^{-1},c_k^{-1};q)_{n_{k-1}}(b_1 c_1)^{n_0} \cdots (b_kc_k)^{n_{k-1}}
\end{align}
also form a Bailey pair relative to $(t,q)$. This comes from iterations of \emph{Bailey's lemma} \cite[pp.~25--26, Theorem~3.3]{And1986}. Finally, the \emph{conjugate Bailey pair} $(\gamma_n,\delta_n)$ relative to $(t,q)$ is such that
\begin{align}\label{eq:Bailey-conj-gamma}
	\gamma_n = \sum_{l\ge n} \frac{\delta_l}{(q;q)_{l-n} (tq; q)_{l+n}}.
\end{align}
Under different choices of Bailey and conjugate Bailey pairs, we are taken to a realm of summation formulas, with Bailey's transform recalled.

\section{New conjugate Bailey pair}\label{sec:conjugate-Bailey}

Conjugate Bailey pairs are not as well explored as Bailey pairs; see the work of Schilling and Warnaar~\cite{SW2002} as an instance. Now our task is to establish the following new conjugate Bailey pair.

\begin{theorem}\label{th:conjugate-Bailey}
	The sequences
	\begin{align}
		\gamma_n = \frac{t^n (q;q)_{2n} (t^2;q)_\infty}{(t^2;q)_{2n} (t,tq,tz,tz^{-1};q)_\infty} \sum_{j=0}^{2n} \frac{(t;q)_j (t;q)_{2n-j}}{(q;q)_j (q;q)_{2n-j}}z^{j-n}
	\end{align}
	and
	\begin{align}
		\delta_n = t^n \sum_{j=0}^{2n} \qbinom{2n}{j}_q z^{j-n}
	\end{align}
	form a conjugate Bailey pair relative to $(t,q)$.
\end{theorem}

\begin{proof}
	Let $z:=\ee^{\ii \theta}$. Recall the \emph{continuous $q$-Hermite polynomials} \cite[p.~320, eq.~(13.1.7)]{Ism2009} (or essentially the \emph{Rogers--Szeg\H{o} polynomials}; see \cite[p.~49, Example~3]{And1998}):
	\begin{align}\label{eq:q-Hermite-def}
		H_n(z;q) := \sum_{j=0}^n \qbinom{n}{j}_q z^{n-2j},
	\end{align}
	and the \emph{continuous $q$-ultraspherical polynomials} \cite[p.~326, eq.~(13.2.1)]{Ism2009}:
	\begin{align}\label{eq:q-ultraspherical-def}
		C_n(z,t;q) := \sum_{j=0}^n \frac{(t;q)_j (t;q)_{n-j}}{(q;q)_j (q;q)_{n-j}} z^{n-2j}.
	\end{align}
	
	It is a standard result that these continuous $q$-Hermite polynomials form an orthogonal basis of symmetric Laurent polynomials in $z$. As such, we may assume that $c_{n,l}$ are coefficients in the infinite expansion of the symmetric bilateral even series
	\begin{align}\label{eq:C-expansion}
		\frac{C_{2n}(z,t;q)}{(t z^2, t z^{-2};q)_\infty} = \sum_{l\ge 0} c_{n,l}\,  H_{2l}(z;q).
	\end{align}
	In particular, we do not have odd-indexed $q$-Hermite polynomials in this expansion because they are odd functions in $z$. Recall the \emph{orthogonality} of continuous $q$-Hermite polynomials \cite[p.~320, Theorem~13.1.3]{Ism2009}:
	\begin{align*}
		\int_0^\pi H_n(z;q) H_{n'}(z;q) (z^2, z^{-2};q)_\infty \dd\theta = \frac{2\pi (q; q)_n}{(q; q)_\infty} \delta_{n,n'},
	\end{align*}
	where $\delta_{n,n'}$ is the \emph{Kronecker delta}. Therefore,
	\begin{align*}
		c_{n,l} = \frac{(q;q)_\infty}{2 (q;q)_{2l}} \int_0^\pi C_{2n}(z,t;q) H_{2l}(z;q) \frac{(z^2, z^{-2};q)_\infty}{(t z^2, t z^{-2};q)_\infty} \frac{\dd\theta}{\pi}.
	\end{align*}
	Using \eqref{eq:q-Hermite-def} for $H_{2l}(z;q)$, we further write $c_{n,l}$ as
	\begin{align}\label{eq:c-expression-int}
		c_{n,l} = \frac{(q;q)_\infty}{2 (q;q)_{2l}} \sum_{j=0}^{2l} \qbinom{2l}{j}_q \int_0^\pi z^{2l-2j} C_{2n}(z,t;q) \frac{(z^2, z^{-2};q)_\infty}{(t z^2, t z^{-2};q)_\infty} \frac{\dd\theta}{\pi}.
	\end{align}
	
	For the integral
	\begin{align}\label{eq:Imn-def}
		I_{m,n} := \int_0^\pi z^{-2 (m+n)} C_{2n}(z,t;q) \frac{(z^2, z^{-2};q)_\infty}{(t z^2, t z^{-2};q)_\infty} \frac{\dd\theta}{\pi},
	\end{align}
	we apply \eqref{eq:q-ultraspherical-def} for $C_{2n}(z,t;q)$ and then invoke \emph{Ramanujan's ${}_{1} \psi_1$ sum} \cite[p.~357, eq.~(II.29)]{GR2004}:
	\begin{align*}
		\sum_{k=-\infty}^\infty \frac{(a;q)_k z^k}{(b;q)_k} = \frac{(q,b/a,az,q/(az);q)_\infty}{(b,q/a,z,b/(az);q)_\infty}
	\end{align*}
	with $(a,b,z)\mapsto (t^{-1},t,tz^2)$ so that
	\begin{align*}
		\frac{(z^2, z^{-2};q)_\infty}{(t z^2, t z^{-2};q)_\infty} = \frac{(t,tq;q)_\infty (1-z^{-2})}{(q,t^2;q)_\infty} \sum_{k=-\infty}^{\infty} \frac{(t^{-1};q)_k t^k z^{2k}}{(t;q)_k}.
	\end{align*}
	Hence, $I_{m,n}$ equals
	\begin{align*}
		\frac{(t,tq;q)_\infty}{(q,t^2;q)_\infty} \sum_{j=0}^{2n} \frac{(t;q)_j (t;q)_{2n-j}}{(q;q)_j (q;q)_{2n-j}} \sum_{k=-\infty}^{\infty} \frac{(t^{-1};q)_k t^k}{(t;q)_k} \int_0^\pi z^{-2 (m+j-k)} (1-z^{-2}) \frac{\dd\theta}{\pi},
	\end{align*}
	yielding the relation
	\begin{align}\label{eq:I-expression}
		I_{m,n} = \frac{(t,tq;q)_\infty}{(q,t^2;q)_\infty} \big(S_{m,n} - S_{m+1,n}\big),
	\end{align}
	where
	\begin{align*}
		S_{d,n} := \sum_{j=0}^{2n} \frac{(t;q)_j (t;q)_{2n-j} (t^{-1};q)_{j+d} t^{j+d}}{(q;q)_j (q;q)_{2n-j} (t;q)_{j+d}}.
	\end{align*}
	
	Now we write $S_{d,n}$ in terms of a ${}_3 \phi_{2}$ series
	\begin{align*}
		S_{d,n} = \frac{(t;q)_{2n} (t^{-1};q)_d t^d}{(q;q)_{2n} (t;q)_d}\, {}_3 \phi_{2}\left(\begin{matrix}
			t,t^{-1}q^d,q^{-2n}\\
			tq^d,t^{-1}q^{1-2n}
		\end{matrix};q,q\right),
	\end{align*}
	and further evaluate it as
	\begin{align}\label{eq:S-expression}
		S_{d,n} = \frac{(t^2;q)_{2n} (q^{d};q)_{2n} (t^{-1};q)_d t^d}{(q;q)_{2n} (t;q)_{2n+d}},
	\end{align}
	where we have applied the \emph{$q$-Pfaff--Saalsch\"utz sum} \cite[p.~355, eq.~(II.12)]{GR2004}:
	\begin{align*}
		{}_3 \phi_{2}\left(\begin{matrix}
			a,b,q^{-n}\\
			c,abq^{1-n}/c
		\end{matrix};q,q\right) = \frac{(c/a,c/b;q)_n}{(c,c/(ab);q)_n}.
	\end{align*}
	
	Recall from \eqref{eq:c-expression-int} and \eqref{eq:I-expression} that
	\begin{align*}
		c_{n,l} &= \frac{(q;q)_\infty}{2 (q;q)_{2l}} \sum_{j=0}^{2l} \qbinom{2l}{j}_q I_{j-l-n,n}\\
		&= \frac{(t,tq;q)_\infty}{2(t^2;q)_\infty (q;q)_{2l}} \sum_{j=0}^{2l} \qbinom{2l}{j}_q \big(S_{j-l-n,n} - S_{j-l-n+1,n}\big).
	\end{align*}
	For every $j$ with $0\le j\le 2l$, write $j' := 2l-j$. Our key observation is that for $j$ in this range, it is always true that
	\begin{align*}
		\qbinom{2l}{j}_q S_{j-l-n,n} = - \qbinom{2l}{j'}_q S_{j'-l-n+1,n}.
	\end{align*} 
	Sum the above for $j$ with $0\le j\le 2l$ so that $j'$ also runs from $0$ to $2l$. Thus,
	\begin{align}\label{eq:S-symmetry}
		\sum_{j=0}^{2l} \qbinom{2l}{j}_q S_{j-l-n,n} = - \sum_{j=0}^{2l}  \qbinom{2l}{j}_q S_{j-l-n+1,n}.
	\end{align}
	Consequently,
	\begin{align*}
		c_{n,l} &= \frac{(t,tq;q)_\infty}{(t^2;q)_\infty (q;q)_{2l}} \sum_{j=0}^{2l} \qbinom{2l}{j}_q S_{j-l-n,n}\\
		&= \frac{(t,tq;q)_\infty (t^2;q)_{2n}}{(t^2;q)_\infty (q;q)_{2n}} \sum_{j=0}^{2l} \frac{(q^{j-l-n};q)_{2n} (t^{-1};q)_{j-l-n} t^{j-l-n}}{(q;q)_j (q;q)_{2l-j} (t;q)_{j-l+n}},
	\end{align*}
	where \eqref{eq:S-expression} has been utilized. It then follows from \eqref{eq:key-id-1} that
	\begin{align*}
		c_{n,l} = \frac{(t,tq;q)_\infty (t^2;q)_{2n} t^{l-n}}{(t^2;q)_\infty (q;q)_{2n} (q;q)_{l-n} (tq;q)_{l+n}}.
	\end{align*}
	In particular,
	\begin{align*}
		c_{n,l} = 0, \qquad (0\le l\le n-1).
	\end{align*}
	
	Finally, substituting these $c_{n,l}$ into \eqref{eq:C-expansion}, we have
	\begin{align*}
		\frac{t^n (q;q)_{2n} (t^2;q)_\infty}{(t^2;q)_{2n} (t,tq,tz^2,tz^{-2};q)_\infty} C_{2n}(z,t;q) = \sum_{l\ge n} \frac{t^l}{(q;q)_{l-n} (tq;q)_{l+n}} H_{2l}(z;q).
	\end{align*}
	Recalling \eqref{eq:Bailey-conj-gamma}, this exactly confirms the claimed conjugate Bailey pair under the change of variables $z \mapsto z^{-1/2}$.
\end{proof}

\begin{corollary}
	For any Bailey pair $(\alpha_n,\beta_n)$ relative to $(t,q)$, we have
	\begin{align}\label{eq:Bailey-special}
		&\sum_{n\ge 0} \beta_n\cdot t^n \sum_{j=0}^{2n} \qbinom{2n}{j}_q z^{j-n}\notag\\
		&\qquad = \frac{(t^2;q)_\infty}{(t,tq,tz,tz^{-1};q)_\infty} \sum_{n\ge 0} \alpha_n \cdot \frac{t^n (q;q)_{2n}}{(t^2;q)_{2n}} \sum_{j=0}^{2n} \frac{(t;q)_j (t;q)_{2n-j}}{(q;q)_j (q;q)_{2n-j}}z^{j-n}.
	\end{align}
\end{corollary}

\begin{proof}
	It follows from Bailey's transform \eqref{eq:Bailey-trans} with the conjugate Bailey pair in Theorem~\ref{th:conjugate-Bailey} on call.
\end{proof}

\section{Fermionic--bosonic duality}\label{sec:fermionic-bosonic}

For the moment, we are ready to close our proof of the fermionic--bosonic duality in Theorem~\ref{th:fermionic-bosonic-main}. Starting from the most elementary Bailey pair relative to $(t,q)$ \cite[p.~31, eqs.~(3.47) and (3.48)]{And1986}:
\begin{align*}
	\alpha_n^{(0)} = \frac{(-1)^n q^{\binom{n}{2}} (1-t q^{2n}) (t;q)_n}{(1-t) (q;q)_n}, \qquad\qquad\qquad
	\beta_n^{(0)} = \delta_{n,0},
\end{align*}
we produce a chain of Bailey pairs in light of \eqref{eq:Bailey-chain-k-a} and \eqref{eq:Bailey-chain-k-b}:
\begin{align*}
	\alpha_n^{(1)} &= \frac{(-1)^n t^n q^{n+\binom{n}{2}} (1-t q^{2n}) (t;q)_n}{(1-t) (q;q)_n}\cdot \frac{(b_1^{-1},c_1^{-1};q)_n  (b_1 c_1)^n}{(b_1tq,c_1tq;q)_n},\\
	\beta_n^{(1)} &= \frac{1}{(q;q)_n}\cdot \frac{(b_1c_1tq;q)_n}{(b_1tq,c_1tq;q)_n},
\end{align*}
and for $k\ge 2$,
\begin{align*}
	\alpha_n^{(k)} &= \frac{(-1)^n t^{kn} q^{kn+\binom{n}{2}} (1-t q^{2n}) (t;q)_n}{(1-t) (q;q)_n}\\
	&\quad \times \frac{(b_1^{-1},c_1^{-1},\ldots,b_k^{-1},c_k^{-1};q)_n (b_1c_1\cdots b_kc_k)^n}{(b_1tq,c_1tq,\ldots,b_ktq,c_ktq;q)_n},\\
	\beta_n^{(k)} &= \sum_{0 \le n_1\le \cdots\le n_{k-1}\le n} \frac{(tq)^{n_1+\cdots+n_{k-1}}}{(q;q)_{n-n_{k-1}}(q;q)_{n_{k-1}-n_{k-2}}\cdots (q;q)_{n_2-n_1} (q;q)_{n_1}}\\
	&\quad\times \frac{(b_kc_ktq;q)_{n-n_{k-1}}(b_{k-1}c_{k-1}tq;q)_{n_{k-1}-n_{k-2}}\cdots (b_2c_2tq;q)_{n_2-n_1} (b_1c_1tq;q)_{n_1}}{(b_1tq,c_1tq;q)_{n_1}\cdots (b_{k-1}tq,c_{k-1}tq;q)_{n_{k-1}}(b_{k}tq,c_{k}tq;q)_{n}}\notag\\
	&\quad\times (b_2^{-1},c_2^{-1};q)_{n_1}\cdots (b_k^{-1},c_k^{-1};q)_{n_{k-1}} (b_2 c_2)^{n_1} \cdots (b_kc_k)^{n_{k-1}}.
\end{align*}
Apply them to \eqref{eq:Bailey-special} and make the substitution $z\mapsto z^{2}$. We immediately arrive at the following strengthening of Theorem~\ref{th:fermionic-bosonic-main}, which reduces to \eqref{eq:fermionic-bosonic-main} at the limiting case where all $b$'s and $c$'s go to zero.

\begin{theorem}\label{th:fermionic-bosonic-main-generalization}
	For every $k\ge 1$,
	\begin{align}
		&\sum_{n_k\ge \cdots\ge n_1 \ge 0} \frac{t^{n_1+\cdots+n_{k}} q^{n_1+\cdots+n_{k-1}}}{(q;q)_{n_k-n_{k-1}}\cdots (q;q)_{n_2-n_1} (q;q)_{n_1}} \sum_{j=0}^{2n_k} \qbinom{2n_k}{j}_q z^{2j-2n_k}\notag\\
		&\times \frac{(b_kc_ktq;q)_{n_k-n_{k-1}}\cdots (b_2c_2tq;q)_{n_2-n_1} (b_1c_1tq;q)_{n_1}}{(b_1tq,c_1tq;q)_{n_1}\cdots (b_{k}tq,c_{k}tq;q)_{n_k}}\notag\\
		&\times (b_2^{-1},c_2^{-1};q)_{n_1}\cdots (b_k^{-1},c_k^{-1};q)_{n_{k-1}} (b_2 c_2)^{n_1} \cdots (b_kc_k)^{n_{k-1}}\notag\\
		&\qquad = \frac{1}{(t,tz^2,tz^{-2};q)_\infty}\sum_{n\ge 0} (-1)^n t^{(k+1)n} q^{kn+\binom{n}{2}} \frac{(q^{n+1};q)_n (t^2q^{2n};q)_\infty}{(tq^n;q)_n (tq^{2n+1};q)_\infty}\notag\\
		&\qquad\quad\times \sum_{j=0}^{2n} \frac{(t;q)_j (t;q)_{2n-j} z^{2j-2n}}{(q;q)_j (q;q)_{2n-j}} \cdot \frac{(b_1^{-1},c_1^{-1},\ldots,b_k^{-1},c_k^{-1};q)_n (b_1c_1\cdots b_kc_k)^n}{(b_1tq,c_1tq,\ldots,b_ktq,c_ktq;q)_n}.
	\end{align}
\end{theorem}

\section{A second fermionic formula}\label{sec:KKS-2.26}

The goal of this section is to show the equivalence between \eqref{eq:ABBST-z} and \eqref{eq:KKS-2.26} as stated in Theorem~\ref{th:KKS-2.26-id}. To begin with, we single out summations over $r$'s on the right-hand side of \eqref{eq:KKS-2.26-id}. In particular, for each $i$ with $1\le i\le k$,
\begin{align*}
	\sum_{r_i\ge 0} \frac{q^{r_i(s_{i-1}+s_i+1)}}{(t,q;q)_{r_i}} &= \lim_{\tau\to 0}{}_{2} \phi_{1} \left(\begin{matrix}
		0,\tau\\
		t
	\end{matrix};q,q^{s_{i-1}+s_i+1}\right)\\
	&= \lim_{\tau\to 0} \frac{(\tau,0;q)_\infty}{(t,q^{s_{i-1}+s_i+1};q)_\infty} {}_{2} \phi_{1} \left(\begin{matrix}
		t/\tau,q^{s_{i-1}+s_i+1}\\
		0
	\end{matrix};q,\tau\right),
\end{align*}
where we have applied \emph{Heine's first transformation} \cite[p.~359, eq.~(III.1)]{GR2004}:
\begin{align*}
	{}_{2} \phi_{1} \left(\begin{matrix}
		a,b\\
		c
	\end{matrix};q,z\right) = \frac{(b,az;q)_\infty}{(c,z;q)_\infty} {}_{2} \phi_{1} \left(\begin{matrix}
		c/b,z\\
		az
	\end{matrix};q,b\right).
\end{align*}
Thus,
\begin{align*}
	(t,q;q)_\infty \sum_{r_i\ge 0} \frac{q^{r_i(s_{i-1}+s_i+1)}}{(t,q;q)_{r_i}} = \sum_{r_i\ge 0} \frac{(-1)^{r_i} t^{r_i} q^{\binom{r_i}{2}} (q;q)_{s_{i-1}+s_i+r_i}}{(q;q)_{r_i}}.
\end{align*}
Before substituting the above relations into the right-hand side of \eqref{eq:KKS-2.26-id}, we change the indices for each $i$ with $1\le i\le k$,
\begin{align*}
	n_i := r_i + s_i.
\end{align*}
It follows that
\begin{align}\label{eq:KKS-2.26-middle-step}
	\RHS\eqref{eq:KKS-2.26-id} = \sum_{n_1,\ldots,n_k\ge 0} t^{\sum_{i=1}^k n_i} (q;q)_{n_1} B_{n_k}(z;q) \prod_{i=1}^{k-1} \Phi_{n_i,n_{i+1}}(q),
\end{align}
where
\begin{align*}
	B_{n}(z;q) := \sum_{s = 0}^n \frac{(-1)^{n-s} q^{\binom{n-s}{2}}}{(q;q)_{s}^2 (q;q)_{n-s}}\sum_{u_1,u_2\ge 0} \qbinom{s}{u_1}_q \qbinom{s}{u_2}_q z^{2u_1-2u_2},
\end{align*}
and
\begin{align*}
	\Phi_{n,n'}(q) := \sum_{s = 0}^n \frac{(-1)^{n-s} q^{\binom{n-s}{2}}(q;q)_{s+n'}}{(q;q)_{s}^2(q;q)_{n-s}}.
\end{align*}

For $B_{n}(z;q)$, we write it in terms of the continuous $q$-Hermite polynomial \eqref{eq:q-Hermite-def}:
\begin{align*}
	B_n(z;q) = \sum_{s = 0}^n \frac{(-1)^{n-s} q^{\binom{n-s}{2}}}{(q;q)_{s}^2 (q;q)_{n-s}} H_{s}(z;q)^2,
\end{align*}
where we have noted the symmetry $H_n(z;q) = H_n(z^{-1};q)$. Now we require the linearization of products of $H_n(z;q)$ \cite[p.~322, Theorem~13.1.5]{Ism2009}:
\begin{align*}
	H_m(z;q) H_n(z;q) = \sum_{l = 0}^{\min(m,n)} \frac{(q;q)_m (q;q)_n}{(q;q)_l (q;q)_{m-l} (q;q)_{n-l}} H_{m+n-2l}(z;q).
\end{align*}
It follows that
\begin{align*}
	B_n(z;q) = \sum_{s = 0}^n \frac{(-1)^{n-s} q^{\binom{n-s}{2}}}{(q;q)_{s}^2 (q;q)_{n-s}} \sum_{l = 0}^s \frac{(q;q)_s^2}{(q;q)_l (q;q)_{s-l}^2} H_{2s-2l}(z;q).
\end{align*}
Substituting $l\mapsto s-j$ and interchanging the summations, we have
\begin{align*}
	B_n(z;q) = \sum_{j = 0}^n \frac{H_{2j}(z;q)}{(q;q)_j^2} \sum_{s = j}^n \frac{(-1)^{n-s} q^{\binom{n-s}{2}}}{(q;q)_{n-s} (q;q)_{s-j}}.
\end{align*}
For the inner sum, we further make the change of indices $s\mapsto s+j$ and write it as a ${}_{1} \phi_{0}$ series so that
\begin{align*}
	B_n(z;q) = \sum_{j = 0}^n \frac{(-1)^{n-j} q^{\binom{n}{2}+\binom{j}{2}+j(1-n)} H_{2j}(z;q)}{(q;q)_j^2} {}_{1} \phi_{0} \left(\begin{matrix}
		q^{-(n-j)}\\
		-
	\end{matrix};q,q\right),
\end{align*}
while this ${}_{1} \phi_{0}$ series becomes the Kronecker delta $\delta_{j,n}$ according to the \emph{$q$-binomial theorem} \cite[p.~354, eq.~(II.4)]{GR2004}:
\begin{align*}
	{}_{1} \phi_{0} \left(\begin{matrix}
		q^{-n}\\
		-
	\end{matrix};q,z\right) = (zq^{-n};q)_n.
\end{align*}
Therefore, by further using the symmetry $H_n(z;q) = H_n(z^{-1};q)$,
\begin{align}\label{eq:B-eva}
	B_n(z;q) = \frac{H_{2n}(z;q)}{(q;q)_n^2} = \frac{1}{(q;q)_n^2} \sum_{j=0}^{2n} \qbinom{2n}{j}_q z^{2j-2n}.
\end{align}

For $\Phi_{n,n'}(q)$, we have
\begin{align*}
	\Phi_{n,n'}(q) &= \frac{(-1)^n q^{\binom{n}{2}} (q;q)_{n'}}{(q;q)_n} {}_{2} \phi_{1} \left(\begin{matrix}
		q^{n'+1},q^{-n}\\
		q
	\end{matrix};q,q\right)\\
	&= \frac{(-1)^n q^{\binom{n}{2}} (q;q)_{n'}}{(q;q)_n} \cdot \frac{q^{n(n'+1)} (q^{-n'};q)_n}{(q;q)_n},
\end{align*}
where for the evaluation of the ${}_{2} \phi_{1}$ series, we have used the \emph{second $q$-Chu--Vandermonde sum} \cite[p.~354, eq.~(II.6)]{GR2004}:
\begin{align*}
	{}_{2} \phi_{1} \left(\begin{matrix}
		a,q^{-n}\\
		c
	\end{matrix};q,q\right) = \frac{a^n (c/a;q)_n}{(c;q)_n}.
\end{align*}
Hence,
\begin{align}\label{eq:Phi-eva}
	\Phi_{n,n'}(q) = \frac{q^{n^2}(q;q)_{n'}}{(q;q)_{n}} \qbinom{n'}{n}_q.
\end{align}

Finally, applying \eqref{eq:B-eva} and \eqref{eq:Phi-eva} to \eqref{eq:KKS-2.26-middle-step}, we find that
\begin{align*}
	\RHS\eqref{eq:KKS-2.26-id} = \sum_{n_k\ge \cdots\ge n_1 \ge 0} \frac{t^{n_1+\cdots+n_{k}} q^{n_1^2+\cdots + n_{k-1}^2}}{(q;q)_{n_k-n_{k-1}}\cdots (q;q)_{n_2-n_1} (q;q)_{n_1}}\sum_{j=0}^{2n_k} \qbinom{2n_k}{j}_q z^{2j - 2n_k},
\end{align*}
as claimed.

\section{Closing remarks}\label{sec:closing}

In \cite{And2001}, Andrews extended the scheme of Bailey pairs to a well-poised version. To be specific, sequences $\alpha'_n$ and $\beta'_n$ form a \emph{WP-Bailey pair} relative to $(s,t,q)$ \cite[p.~15, eq.~(6.2)]{And2001} if for each $n\ge 0$,
\begin{align}\label{eq:WP-Bailey-beta}
	\beta'_n = \sum_{l=0}^n \frac{(st^{-1};q)_{n-l} (s; q)_{n+l}}{(q;q)_{n-l} (tq; q)_{n+l}}\,\alpha'_l;
\end{align}
likewise, sequences $\gamma'_n$ and $\delta'_n$ form a \emph{conjugate WP-Bailey pair} relative to $(s,t,q)$ if for each $n\ge 0$,
\begin{align}\label{eq:WP-Bailey-conj-gamma}
	\gamma'_n = \sum_{l\ge n} \frac{(st^{-1};q)_{l-n} (s;q)_{l+n}}{(q;q)_{l-n} (tq; q)_{l+n}}\,\delta'_l.
\end{align}
It is notable that such a well-poised extension reduces to the original form by taking $s=0$.

Toward this direction, we may lift the conjugate Bailey pair in Theorem~\ref{th:conjugate-Bailey} to the following WP-extension. This conjugate WP-Bailey pair further leads us to a generalization of the relation in Theorem~\ref{th:fermionic-bosonic-main-generalization} by applying the WP-Bailey pairs acquired from iterations of \cite[pp.~15--16, Theorem~7]{And2001} starting with the seed case \cite[p.~17, eqs.~(6.6) and (6.7)]{And2001}; we will not record this generalized identity due to its oversized expression. It remains unknown to understand the physical meaning conveyed by such a generalization, especially the new parameter $s$.

\begin{theorem}\label{th:conjugate-WP-Bailey}
	The sequences
	\begin{align}
		\gamma'_n = \frac{t^n (q;q)_{2n} (t^2, sz, sz^{-1};q)_\infty}{(t^2;q)_{2n} (t,tq,tz,tz^{-1};q)_\infty} \sum_{j=0}^{2n} \frac{(t;q)_j (t;q)_{2n-j}}{(q;q)_j (q;q)_{2n-j}}z^{j-n}
	\end{align}
	and
	\begin{align}
		\delta'_n = \frac{t^n (1-s q^{2n}) (q;q)_{2n} (s^2;q)_\infty}{(1-s) (s^2;q)_{2n} (s,sq;q)_\infty} \sum_{j=0}^{2n} \frac{(s;q)_j (s;q)_{2n-j}}{(q;q)_j (q;q)_{2n-j}}z^{j-n}
	\end{align}
	form a conjugate WP-Bailey pair relative to $(s,t,q)$.
\end{theorem}

\begin{proof}
	See Appendix~\ref{sec:conjugate-WP-Bailey}.
\end{proof}

\subsection*{Acknowledgements}

Shane Chern was supported by the Austrian Science Fund (no.~10.55776/F1002). We are grateful to Ranveer Kumar Singh for informing us of their work \cite{ABST2025}, and to George Andrews, Matthew Buican, and Ole Warnaar for useful feedback. We also acknowledge conversations with GPT-5.5 Pro that suggested the bosonic formula for the Macdonald index.

\bibliographystyle{amsplain}

\appendix

\section{Original formulation of the fermionic sum in \eqref{eq:KKS-2.26}}\label{appx:fer-2-original}

Here we examine the equivalence between the fermionic sum in \eqref{eq:KKS-2.26} and that in Andrews et al.~\cite{ABST2025} or Kim et al.~\cite{KKS2026}. Note that \cite[p.~8, eq.~(2.26)]{KKS2026}, with $z$ rescaled to $z^2$ and $T$ substituted by $tq^{-1}$, reads:
\begin{align}\label{eq:fer-2-original}
	&(t,q;q)_\infty^{k} \sum_{\substack{l_1,\ldots,l_{2k+1}\ge 0\\m_1,\ldots,m_{2k+1}\ge 0}} \delta_{l_{2k}+l_{2k+1},m_{2k}+m_{2k+1}} \prod_{i=1}^{2k-1} \delta_{l_i,m_i}\notag\\
	&\times \frac{q^{\sum_{i,j=1}^{2k+1} \frac{a_{i,j} l_i m_j}{2} + \sum_{i=1}^{k} \frac{l_{2i-1}+m_{2i-1}}{2}} t^{\sum_{i=1}^{k}\frac{l_{2i}+m_{2i}}{2}+\frac{l_{2k+1}+m_{2k+1}}{2}} z^{2m_{2k+1}-2l_{2k+1}}}{(q;q)_{l_{2k+1}}(q;q)_{m_{2k+1}} \prod_{i=1}^k (q;q)_{l_{2i}}(q;q)_{m_{2i}}(t;q)_{l_{2i-1}}(q;q)_{m_{2i-1}}},
\end{align}
where $A := (a_{i,j})_{1\le i,j\le 2k+1}$ is given by $A = 2I_{2k+1} - M_{\mathsf{Car}}(D_{2k+1})$ with $I_{2k+1}$ the identity matrix of dimension $2k+1$ and $M_{\mathsf{Car}}(D_{2k+1})$ the Cartan matrix of the $D_{2k+1}$ Lie algebra. In particular, in this case $A$ is the adjacency matrix for the $D_{2k+1}$ Dynkin diagram: 
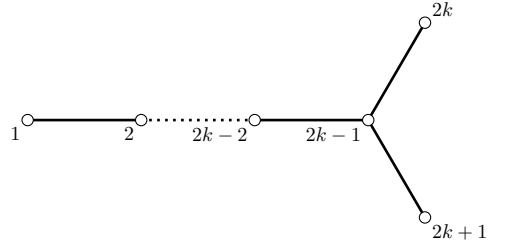
\begin{align*}
	\vcenter{\hbox{\begin{tikzpicture}[scale=1.5]
				\draw[line width = 1pt] (0,0) -- (1,0);
				\draw[dotted, line width = 1pt] (1,0) -- (2,0);
				\draw[line width = 1pt] (2,0) -- (3,0);
				\draw[line width = 1pt] (3,0) -- (3.5,0.86);
				\draw[line width = 1pt] (3,0) -- (3.5,-0.86);
				\foreach\i in {0, 1, 2, 3} {
					\draw[black,fill=white] (\i,0) circle (.05);
				}
				\draw[black,fill=white] (3.5,0.86) circle (.05);
				\draw[black,fill=white] (3.5,-0.86) circle (.05);
				\node[below left, scale=0.75] at (0,0) {$1$};
				\node[below left, scale=0.75] at (1,0) {$2$};
				\node[below left, scale=0.75] at (2,0) {$2k-2$};
				\node[below left, scale=0.75] at (3,0) {$2k-1$};
				\node[above right, scale=0.75] at (3.5,0.86) {$2k$};
				\node[below right, scale=0.75] at (3.5,-0.86) {$2k+1$};
	\end{tikzpicture}}}\quad
	,
\end{align*}
so that $a_{i,i+1} = a_{i+1,i} = 1$ for $i$ with $1\le i\le 2k-2$, $a_{2k-1,2k} = a_{2k,2k-1} = a_{2k-1,2k+1} = a_{2k+1,2k-1} = 1$, and $a_{i,j} = 0$ for all other cases.

In light of the Kronecker delta factors in \eqref{eq:fer-2-original}, we rename the summation indices as in \eqref{eq:KKS-2.26}:
\begin{align*}
	r_i &:= l_{2i-1} = m_{2i-1}, &&\quad (1\le i\le k),\\
	s_i &:= l_{2i} = m_{2i}, &&\quad (1\le i\le k-1),
\end{align*}
and
\begin{align*}
	s_k := l_{2k}+l_{2k+1} = m_{2k}+m_{2k+1}.
\end{align*}
Thus, the exponent of $q$ in the numerator becomes
\begin{align*}
	\sum_{i,j=1}^{2k+1} \frac{a_{i,j} l_i m_j}{2} + \sum_{i=1}^{k} \frac{l_{2i-1}+m_{2i-1}}{2} = \sum_{i=1}^k r_i (s_{i-1} + s_i + 1),
\end{align*}
while the exponent of $t$ is
\begin{align*}
	\sum_{i=1}^{k}\frac{l_{2i}+m_{2i}}{2}+\frac{l_{2k+1}+m_{2k+1}}{2}  = \sum_{i=1}^k s_i.
\end{align*}
We further put
\begin{align*}
	u_1 := l_{2k}, \qquad u_2 := m_{2k},
\end{align*}
so that
\begin{align*}
	l_{2k+1} = s_k - u_1, \qquad m_{2k+1} = s_k - u_2.
\end{align*}
In this way, the exponent of $z$ becomes
\begin{align*}
	2m_{2k+1}-2l_{2k+1} = 2u_1 - 2u_2.
\end{align*}
Finally, for the denominator, we have
\begin{align*}
	\frac{1}{(t,q;q)_{r_1}\cdots (t,q;q)_{r_k} (q;q)_{s_1}^2 \cdots (q;q)_{s_{k-1}}^2}\cdot \frac{1}{(q;q)_{l_{2k}}(q;q)_{m_{2k}}(q;q)_{l_{2k+1}}(q;q)_{m_{2k+1}}},
\end{align*}
while the latter factor can be further rewritten as
\begin{align*}
	\frac{1}{(q;q)_{l_{2k}+l_{2k+1}} (q;q)_{m_{2k}+m_{2k+1}}} \qbinom{l_{2k}+l_{2k+1}}{l_{2k}}_q \qbinom{m_{2k}+m_{2k+1}}{m_{2k}}_q = \frac{1}{(q;q)_{s_k}^2} \qbinom{s_k}{u_1}_q \qbinom{s_k}{u_2}_q.
\end{align*}
Now \eqref{eq:fer-2-original} becomes exactly the same as the fermionic sum in \eqref{eq:KKS-2.26}.

\section{A basic hypergeometric identity}\label{appx:evaluation}

Here we evaluate the following basic hypergeometric series, which appears as a critical step in our proof of Theorem~\ref{th:conjugate-Bailey}.

\begin{lemma}\label{le:key-id-1}
	For $l,n\ge 0$,
	\begin{align}\label{eq:key-id-1}
		\sum_{j=0}^{2l} \frac{(q^{j-l-n};q)_{2n} (t^{-1};q)_{j-l-n} t^{j}}{(q;q)_j (q;q)_{2l-j} (t;q)_{j-l+n}} = \frac{t^{2l}}{(q;q)_{l-n} (tq;q)_{l+n}}.
	\end{align}
\end{lemma}

\begin{proof}
	If $l< n$, we have the vanishing of $(q^{j-l-n};q)_{2n}$ for all $j$ ranging from $0$ to $2l$, and hence both sides of \eqref{eq:key-id-1} trivially go to zero in this case. Now we assume $l\ge n$. For convenience, let us write
	\begin{align*}
		s_j := \frac{(q^{j-l-n};q)_{2n} (t^{-1};q)_{j-l-n} t^{j}}{(q;q)_j (q;q)_{2l-j} (t;q)_{j-l+n}}, \qquad (0\le j\le 2l).
	\end{align*}
	It is clear that $s_j$ is supported on the ranges $0\le j\le l-n$ and $l+n+1\le j\le 2l$. Putting in addition
	\begin{align*}
		s_{2l+1} := 0,
	\end{align*}
	we observe that for each $j$ with $0\le j\le l-n$,
	\begin{align*}
		s_j + s_{1+2l-j} = \frac{q^j (1-q^{1+2l-2j})}{1-q^{1+2l-j}}\cdot s_j.
	\end{align*}
	Thus,
	\begin{align*}
		\LHS\eqref{eq:key-id-1} = \sum_{j=0}^{l-n} \frac{(q^{j-l-n};q)_{2n} (t^{-1};q)_{j-l-n} t^{j} q^j (1-q^{1+2l-2j})}{(q;q)_j (q;q)_{2l-j} (t;q)_{j-l+n} (1-q^{1+2l-j})}.
	\end{align*}
	Note that
	\begin{align*}
		(q^{j-l-n};q)_{2n} = \frac{(q;q)_{l+n-j}q^{n(2j-2l-1)}}{(q;q)_{l-n-j}}.
	\end{align*}
	We further make the change of indices $j\mapsto l-n-j$ so that
	\begin{align*}
		\LHS\eqref{eq:key-id-1} &= \sum_{j=0}^{l-n} \frac{(q;q)_{2n+j}(t^{-1};q)_{-2n-j} t^{l-n-j} q^{l-2n(n+1)-(2n+1)j} (1-q^{1+2n+2j})}{(q;q)_j (q;q)_{l-n-j} (q;q)_{l+n+j} (t;q)_{-j} (1-q^{1+l+n+j})}\\
		& = \sum_{j=0}^{l-n} \frac{(q;q)_{2n+j}(t^{-1};q)_{-2n-j} t^{l-n-j} q^{l-2n(n+1)-(2n+1)j} (1-q^{1+2n+2j})}{(q;q)_j (q;q)_{l-n-j} (q;q)_{1+l+n+j} (t;q)_{-j}}.
	\end{align*}
	Using the relations
	\begin{align*}
		(a;q)_{m+j} &= (a;q)_m (aq^m;q)_j,\\
		(a;q)_{m-j} &= \frac{(a;q)_m}{(q^{1-m}/a;q)_j} \left(-\frac{q}{a}\right)^j q^{\binom{j}{2}-nj},
	\end{align*}
	we have
	\begin{align*}
		\LHS\eqref{eq:key-id-1} &= \frac{(q;q)_{2n} t^{l+n} q^{l-n}}{(q;q)_{l-n}(q;q)_{l+n+1}(tq;q)_{2n}}\\
		&\quad\times \sum_{j=0}^{l-n} \frac{(q^{2n+1},t^{-1}q,q^{-(l-n)};q)_jt^j q^{(l-n-1)j}(1-q^{1+2n+2j})}{(q,q^{l+n+2},tq^{2n+1};q)_j (-1)^j q^{\binom{j}{2}}}\\
		&= \frac{(q;q)_{2n} t^{l+n} q^{l-n} (1-q^{2n+1})}{(q;q)_{l-n}(q;q)_{l+n+1}(tq;q)_{2n}}\\
		&\quad\times \lim_{\tau\to 0} {}_6 \phi_{5}\left(\begin{matrix}
			q^{2n+1},q^{n+\frac{3}{2}},-q^{n+\frac{3}{2}},t^{-1}q,\tau q^{2n+2},q^{-(l-n)}\\
			q^{n+\frac{1}{2}},-q^{n+\frac{1}{2}},tq^{2n+1},\tau^{-1},q^{l+n+2}
		\end{matrix};q,\frac{tq^{l-n-1}}{\tau}\right).
	\end{align*}
	For this terminating ${}_6 \phi_{5}$ series, we require the following evaluation \cite[p.~356, eq.~(II.21)]{GR2004}:
	\begin{align}\label{eq:6phi5}
		{}_6 \phi_{5}\left(\begin{matrix}
			a, a^{\frac{1}{2}}q, -a^{\frac{1}{2}}q, b, c, q^{-n}\\
			a^{\frac{1}{2}}, -a^{\frac{1}{2}}, aq/b, aq/c, aq^{n+1}
		\end{matrix};q,\frac{aq^{n+1}}{bc}\right) = \frac{(aq,aq/(bc);q)_n}{(aq/b,aq/c;q)_n}.
	\end{align}
	Therefore,
	\begin{align*}
		\LHS\eqref{eq:key-id-1} &= \frac{(q;q)_{2n} t^{l+n} q^{l-n} (1-q^{2n+1})}{(q;q)_{l-n}(q;q)_{l+n+1}(tq;q)_{2n}}\lim_{\tau\to 0} \frac{(q^{2n+2},tq^{-1}\tau^{-1};q)_{l-n}}{(tq^{2n+1},\tau^{-1};q)_{l-n}}\\
		&= \frac{t^{l+n}q^{l-n}}{(q;q)_{l-n} (tq;q)_{l+n}} \lim_{\tau\to 0} \frac{(tq^{-1}\tau^{-1};q)_{l-n}}{(\tau^{-1};q)_{l-n}}.
	\end{align*}
	Finally, using the relation
	\begin{align*}
		\frac{(a;q)_n}{(b;q)_n} = \frac{(q^{1-n}/a;q)_n}{(q^{1-n}/b;q)_n} \left(\frac{a}{b}\right)^n,
	\end{align*}
	the above limit becomes $(tq^{-1})^{l-n}$, producing the right-hand side of \eqref{eq:key-id-1}.
\end{proof}

\section{Lifting to a conjugate WP-Bailey pair}\label{sec:conjugate-WP-Bailey}

Here we show the conjugate WP-Bailey pair claimed in Theorem~\ref{th:conjugate-WP-Bailey}; the proof is similar to that for Theorem~\ref{th:conjugate-Bailey}. In particular, this time we use the fact that the continuous $q$-ultraspherical polynomials $C_n(z,s;q)$ also form an orthogonal basis of symmetric Laurent polynomials in $z$. Then it is possible to write
\begin{align*}
	\frac{(sz^2,sz^{-2};q)_\infty}{(t z^2, t z^{-2};q)_\infty} C_{2n}(z,t;q) = \sum_{l\ge 0} c'_{n,l}\,  C_{2l}(z,s;q).
\end{align*}
Applying the orthogonality of continuous $q$-ultraspherical polynomials \cite[p.~327, Theorem~13.2.1]{Ism2009}:
\begin{align*}
	\int_0^\pi C_n(z,s;q) C_{n'}(z,s;q) \frac{(z^2, z^{-2};q)_\infty}{(s z^2, s z^{-2};q)_\infty} \dd\theta = \frac{2\pi (1-s) (s^2; q)_n (s, sq; q)_\infty}{(1-sq^n) (q; q)_n (q, s^2; q)_\infty} \delta_{n,n'},
\end{align*}
the coefficients $c'_{n,l}$ can be identified by
\begin{align*}
	c'_{n,l} &= \frac{(1-sq^{2l}) (q; q)_{2l} (q, s^2; q)_\infty}{2 (1-s) (s^2; q)_{2l} (s, sq; q)_\infty} \int_0^\pi C_{2n}(z,t;q) C_{2l}(z,s;q) \frac{(z^2, z^{-2};q)_\infty}{(t z^2, t z^{-2};q)_\infty} \frac{\dd\theta}{\pi}\\
	&= \frac{(1-sq^{2l}) (q; q)_{2l} (q, s^2; q)_\infty}{2 (1-s) (s^2; q)_{2l} (s, sq; q)_\infty} \sum_{j=0}^{2l} \frac{(s;q)_j (s;q)_{2l-j}}{(q;q)_j (q;q)_{2l-j}}\, I_{j-l-n,n},
\end{align*}
where $I_{j-l-n,n}$ is as in \eqref{eq:Imn-def}. It follows from \eqref{eq:I-expression} and \eqref{eq:S-expression}, with a similar symmetry to that in \eqref{eq:S-symmetry} used, that
\begin{align*}
	c'_{n,l} &= \frac{t^{-l-n} (1-sq^{2l}) (q; q)_{2l} (t^2;q)_{2n} (t, tq , s^2; q)_\infty}{(1-s) (s^2; q)_{2l} (q;q)_{2n} (t^2, s, sq; q)_\infty}\\
	&\quad\times \sum_{j=0}^{2l} \frac{(s;q)_j (s;q)_{2l-j} (q^{j-l-n};q)_{2n} (t^{-1};q)_{j-l-n} t^{j}}{(q;q)_j (q;q)_{2l-j} (t;q)_{j-l+n}}.
\end{align*}
For the summation on the right-hand side, we proceed in the same way as that for Lemma~\ref{le:key-id-1} by pairing the summands, and eventually reformulate it in terms of a ${}_6 \phi_5$ series:
\begin{align*}
	&\sum_{j=0}^{2l} \frac{(s;q)_j (s;q)_{2l-j} (q^{j-l-n};q)_{2n} (t^{-1};q)_{j-l-n} t^{j}}{(q;q)_j (q;q)_{2l-j} (t;q)_{j-l+n}}\\
	&\qquad = \frac{(s;q)_{l-n} (s;q)_{l+n} (q;q)_{2n} t^{l+n} (1-s^{-1}q) (1-q^{2n+1})}{(q;q)_{l-n}(q;q)_{l+n+1}(tq;q)_{2n} (1-s^{-1}q^{1-(l-n)})}\\
	&\qquad\quad\times {}_6 \phi_{5}\left(\begin{matrix}
		q^{2n+1},q^{n+\frac{3}{2}},-q^{n+\frac{3}{2}},t^{-1}q,s q^{l+n},q^{-(l-n)}\\
		q^{n+\frac{1}{2}},-q^{n+\frac{1}{2}},tq^{2n+1},s^{-1}q^{2-(l-n)},q^{l+n+2}
	\end{matrix};q,\frac{tq}{s}\right).
\end{align*}
Invoking \eqref{eq:6phi5}, we arrive at
\begin{align}\label{eq:c'-sum-evaluation}
	\sum_{j=0}^{2l} \frac{(s;q)_j (s;q)_{2l-j} (q^{j-l-n};q)_{2n} (t^{-1};q)_{j-l-n} t^{j}}{(q;q)_j (q;q)_{2l-j} (t;q)_{j-l+n}} = \frac{(st^{-1};q)_{l-n} (s;q)_{l+n} t^{2l}}{(q;q)_{l-n} (tq;q)_{l+n}}.
\end{align}
Substituting \eqref{eq:c'-sum-evaluation} into the expression for $c'_{n,l}$, we conclude that
\begin{align*}
	c'_{n,l} = \frac{(st^{-1};q)_{l-n} (s;q)_{l+n}}{(q;q)_{l-n} (tq;q)_{l+n}} \cdot \frac{t^{l-n} (1-sq^{2l}) (q;q)_{2l} (t^2;q)_{2n} (t,tq,s^2;q)_\infty}{(1-s)(s^2;q)_{2l}(q;q)_{2n}(t^2,s,sq;q)_\infty},
\end{align*}
which, in particular, vanishes when $l < n$. Therefore,
\begin{align*}
	&\frac{t^n (q;q)_{2n} (t^2,sz^2,sz^{-2};q)_\infty}{(t^2;q)_{2n} (t, tq, t z^2, t z^{-2};q)_\infty} C_{2n}(z,t;q)\\
	&\qquad = \sum_{l\ge n} \frac{(st^{-1};q)_{l-n} (s;q)_{l+n}}{(q;q)_{l-n} (tq;q)_{l+n}} \cdot \frac{t^{l} (1-sq^{2l}) (q;q)_{2l} (s^2;q)_\infty}{(1-s)(s^2;q)_{2l}(s,sq;q)_\infty}  C_{2l}(z,s;q).
\end{align*}
Finally, we replace $z$ with $z^{-1/2}$. The above identity becomes exactly the required relation \eqref{eq:WP-Bailey-conj-gamma} for conjugate WP-Bailey pairs.
	
\end{document}